\documentclass[12pt, a4paper]{article}

\usepackage[utf8]{inputenc}
\usepackage[T1]{fontenc}
\usepackage[english]{babel}

\usepackage{vmargin} 
\setmarginsrb{3cm}{3cm}{3cm}{2cm}{0cm}{0cm}{0cm}{1cm}

\usepackage{hyperref}
\hypersetup{
    pdftitle={Möbius functions of directed restriction species and free operads, via the generalised Rota formula},
    pdfsubject={},
    pdfauthor={Louis Carlier},
    pdfkeywords={},
    colorlinks=true,
    linkcolor=teal,
    citecolor=blue,
    urlcolor=cyan,
    breaklinks=true
}

\usepackage{mathtools, amssymb}
\usepackage{mathrsfs} 

\usepackage{amsthm}
\allowdisplaybreaks
\usepackage{eucal} 
\theoremstyle{plain}
\newtheorem{thm}{Theorem}[section]
\newtheorem{lem}[thm]{Lemma}
\newtheorem{prop}[thm]{Proposition}
\newtheorem{cor}[thm]{Corollary}
\newtheorem*{thm*}{Theorem}
\theoremstyle{definition}

\newtheorem{ex}[thm]{Example}
\theoremstyle{remark}
\newtheorem{rmk}[thm]{Remark}

\usepackage{tikz-cd}

\newcommand{\ds}[1]{\mathbf{#1}}

\providecommand{\kat}[1]{\text{\textbf{\textsl{#1}}}}

\newcommand{\Cat}{\kat{Cat}}
\newcommand{\Grpd}{\kat{Grpd}}
\newcommand{\infGrpd}{{\mathcal{S}}}

\newcommand{\simplexcat}{\boldsymbol \Delta}

\newcommand{\op}{^{\text{{\rm{op}}}}}

\DeclareMathOperator{\Fun}{Fun}

\DeclareMathOperator{\id}{id}

\DeclareMathOperator{\Boxprod}{\Box}

\newcommand{\ora}{\overrightarrow}
\newcommand{\xra}{\xrightarrow}
\newcommand{\xla}{\xleftarrow}

\newcommand{\eq}{\simeq}

\newcommand{\name}[1]{\ulcorner #1\urcorner}

\newcommand{\Q}{\mathbb{Q}}

\newcommand{\Dec}{\operatorname{Dec}}

\newcommand{\arxiv}[1]{\href{http://arxiv.org/abs/#1}{arxiv:#1}}

\author{Louis Carlier
}
\title{Möbius functions of directed restriction species and free operads, via the generalised Rota formula}
\date{}

\newcommand{\address}{{%
  \bigskip
  \footnotesize
  \textsc{Departament de Matemàtiques \\
  \indent Universitat Autònoma de Barcelona \\
  \indent 08193 Bellaterra (Barcelona), Spain}\par\nopagebreak
  \textit{E-mail address}: \href{mailto:louiscarlier@mat.uab.cat}{\nolinkurl{louiscarlier@mat.uab.cat}}
}}

\begin{document}
\maketitle

\begin{abstract}
We present some tools for providing situations where the generalised Rota formula of \cite{Ca} applies. As an example of this, we compute the Möbius function of the incidence algebra of any directed restriction species, free operad, or more generally free monad on a finitary polynomial monad.
\end{abstract}

\phantomsection
\addcontentsline{toc}{section}{Introduction}
\section*{Introduction}
\label{sec:intro}

Decomposition spaces were introduced by Gálvez-Carrillo, Kock, and Tonks \cite{GKT1, GKT2} as a homotopy setting for Möbius inversion, generalising the Möbius categories of Leroux \cite{Leroux76}, in turn a generalisation of locally finite posets \cite{Rota} and Cartier–Foata finite-decomposition monoids \cite{CF}. 
An important tool in the classical setting is Rota's formula which compares the Möbius functions of two posets related by a Galois connection \cite{Rota}.
In \cite{Ca}, a generalisation of this formula was established, generalising first from Galois connections between posets to adjunctions between Möbius categories, then to infinity adjunctions between Möbius decomposition spaces, and in fact shown to hold for relationships weaker than adjunctions, namely certain bisimplicial spaces called bicomodule configurations. 
Any adjunction of categories (e.g.\ a Galois connection of posets) induces a bicomodule configuration, and the formula of \cite{Ca} reduces to the Rota formula in the case of a Galois connection.

The goal of the present paper is to provide situations where the generalised Rota formula can be applied. 
A key point in the construction is the notion of {\em abacus map}, a certain family of extra maps on a bisimplicial groupoid, which allows to modify the vertical top face maps artificially in an interesting way. This construction appears mysterious, but it is justified by the main example, the box product of the decomposition space of finite sets and the décalage of the decomposition space of finite posets: in this case the modification is precisely what allows to apply the generalised Rota formula to compute the Möbius function of any directed restriction species, starting with the case of the incidence algebra of finite posets.
This first result, for the incidence algebra of the decomposition space of finite posets, is known \cite{ABS}, but its derivation from the generalised Rota formula is new and interesting.
The coalgebra of finite posets is the incidence algebra of a decomposition space which is not a category, and the Möbius function is calculated via a bicomodule configuration (which is not an adjunction) with the decomposition space of finite sets.  The construction also yields the Möbius function of the incidence algebra of any directed restriction species, including the Butcher--Connes--Kreimer Hopf algebra, and of the incidence bialgebra of any free operad, or more generally of any free monad
on a finitary polynomial monad.
 
 In Section~\ref{sec:bicomodule}, before defining the bicomodule configuration interpolating between the decomposition space $\ds{C}$ of finite posets and the decomposition space $\ds{I}$ of finite sets, we set up some general theory, introducing the notion of abacus map to modify a bisimplicial groupoid in a useful way.  We furthermore identify conditions needed to obtain the required structure of bicomodule configuration. The constructions are applied to the box product as mentioned, to obtain a bicomodule configuration interpolating between the decomposition space of finite posets and the decomposition space of finite sets. 
 
 In Section~\ref{sec:mobiusfunctions}, it is shown that this bicomodule configuration is Möbius. The verifications are actually elementary and amount essentially to computing some pullbacks of groupoids. Finally the generalised Rota formula can be applied rather easily, yielding a relationship between the two Möbius functions. Since the Möbius function for the coalgebra of finite sets is known, this gives a formula for the Möbius function for the incidence algebra of the decomposition space of finite posets: 

\medskip
\noindent
{\bf Theorem~\ref{formula}.} {\em
    The Möbius function of the incidence algebra of the decomposition space $\ds{C}$ of finite posets is
    \[
    \mu(P) = 
    \begin{cases}
       (-1)^n &\text{ if $P \in \ds{C}_1$ is a discrete poset with $n$ elements}\\
         0    &\text{ else.}
    \end{cases}
    \]
}

\noindent
We show how the result extends almost verbatim to the incidence algebra of any directed restriction species in the sense of \cite{GKT:restrict}, via the decomposition space interpretation.

\medskip
\noindent
{\bf Corollary~\ref{formulaDRS}.} {\em
    The Möbius function of the incidence algebra of the decomposition space $\ds R$ associated to a directed restriction species $R : \mathbb C\op \to \Grpd$, is 
    \[
        \mu(Q) = 
    \begin{cases}
       (-1)^n &\text{ if the underlying poset of $Q \in \ds R_1$ is discrete with $n$ elements}\\
         0    &\text{ else.}
    \end{cases}
    \]
}

\medskip
\noindent
{\bf Corollary~\ref{trees}.} {\em
The Möbius function of the Butcher--Connes--Kreimer Hopf algebra of rooted forests
is
    \[
        \mu(F) = 
    \begin{cases}
       (-1)^n &\text{ if $F$ consists of $n$ isolated root nodes}\\
         0    &\text{ else.}
    \end{cases}
    \]
}

We also obtain a similar function for free operads, or more generally free monads on finitary polynomial monads.

\medskip
\noindent
{\bf Corollary~\ref{mobiusPtrees}.} {\em
    The Möbius function of the incidence bialgebra of $P$-trees (for any finitary polynomial endo\-functor $P$) is
    \[
        \mu(T) = 
    \begin{cases}
       (-1)^n &\text{ if $T$ consists of $n$ $P$-corollas and possibly isolated edges}\\
         0    &\text{ else.}
    \end{cases}
    \]
}

\subsection*{Acknowledgements}
The author would like to thank Joachim Kock for his advice and support all along the project. The author was supported by PhD grant attached to MTM2013-42293-P, and by the FEDER-MINECO grant MTM2016-80439-P from the Spanish Ministry of Economy and Competitiveness.

\section{Set-up}
\label{sec:preliminaries}

\subsection*{Decomposition spaces}
The notion of decomposition space was introduced by Gálvez-Carrillo, Kock, and Tonks \cite{GKT1}, in the general setting of simplicial $\infty$-groupoids, and independently by Dyckerhoff and Kapranov \cite{DK} under the name unital $2$-Segal space.
The natural level of generality for decomposition spaces in combinatorics is that of simplicial groupoids, because many combinatorial objects have symmetries, which are taken care of by the groupoid formalism. For a survey motivated by combinatorics, see \cite{GKT:combinatorics}.

    To state the definition, recall first that the simplex category $\simplexcat$ has an active-inert factorisation system: an arrow is active when it preserves endpoints, and is inert if it is distance preserving.
    In the present contribution, a \emph{decomposition space} $X: \simplexcat\op \to \Grpd$ is a simplicial groupoid such that the image of any active-inert pushout in $\simplexcat$ is a pullback of groupoids.
It is enough to check that the following squares are pullbacks, where $0 \le k \le n$:
  \begin{center}
    \begin{tikzcd}
        X_{n+1}   \arrow[r, "s_{k+1}"] 
                  \arrow[d, "d_\bot"']
                  \arrow[dr, phantom, "\lrcorner", very near start]
          & X_{n+2} \arrow[d, "d_\bot"] \\
        X_n \arrow[r, "s_k"'] & X_{n+1},
    \end{tikzcd} \!\!\!
    \begin{tikzcd}
       X_{n+2} \arrow[d, "d_\bot"] 
          & X_{n+3} \arrow[l, "d_{k+2}"']
                    \arrow[d, "d_\bot"]
                    \arrow[dl, phantom, "\llcorner", very near start] \\
         X_{n+1}  & X_{n+2} \arrow[l, "d_{k+1}"],
    \end{tikzcd} \!\!\!
    \begin{tikzcd}
        X_{n+1}   \arrow[r, "s_k"] 
                  \arrow[d, "d_\top"']
                  \arrow[dr, phantom, "\lrcorner", very near start]
          & X_{n+2} \arrow[d, "d_\top"]  \\
        X_n \arrow[r, "s_k"'] & X_{n+1},
    \end{tikzcd} \!\!\!
    \begin{tikzcd}
        X_{n+2} \arrow[d, "d_\top"] 
          & X_{n+3} \arrow[l, "d_{k+1}"']
                    \arrow[d, "d_\top"]
                    \arrow[dl, phantom, "\llcorner", very near start] \\
         X_{n+1}  & X_{n+2} \arrow[l, "d_{k+1}"].
    \end{tikzcd}
  \end{center}

Decomposition spaces can be seen as an abstraction of posets. 
The pullback condition is precisely the condition required for the span 
\[
    X_1 \xla{d_1} X_2 \xra{(d_2,d_0)} X_1 \times X_1
\] 
to define a counital coassociative comultiplication on $\Grpd_{/X_1}$, called the {\em incidence coalgebra}. See also \cite{Penney} for the exact role played by the decomposition space condition.

\begin{prop}[{\cite[Proposition 2.3.3]{DK}}, {\cite[Proposition 3.5]{GKT1}}]
    Every Segal space is a decomposition space. In particular, the nerve of a poset is a decomposition space.
\end{prop}

A map $f: X \to Y$ of simplicial spaces is \emph{cartesian} on an arrow $[n] \to [k]$ in $\simplexcat$ if the naturality square for $f$ with respect to this arrow is a pullback.
A simplicial map $f: X \to Y$ is \emph{culf} if it is cartesian on all active maps.
The culf functors induce coalgebra homomorphisms between the incidence coalgebras \cite[Lemma 8.2]{GKT1}. 

Recall that a map of groupoids is a \emph{monomorphism} when its fibres are either empty or contractible.
A decomposition space is called \emph{complete} if $s_0: X_0 \to X_1$ is a monomorphism, see \cite[\S 2]{GKT2}.
Under this condition, a general Möbius inversion principle holds \cite{GKT2}.
A complete decomposition space $X$ is called \emph{Möbius} if it is locally finite and of locally finite length \cite[\S 8]{GKT2}. 
The first condition ensures one can take homotopy cardinality to pass from groupoids to $\Q$-vector spaces and obtain the standard incidence coalgebra (as a $\Q$-algebra); the second condition ensures that the Möbius inversion formula admits a cardinality.

Let us recall the definition of décalage \cite{Illusie}. 
Given a simplicial space $X$, the \emph{lower décalage} $\Dec_\bot (X)$ is the simplicial space obtained by deleting $X_0$, all $d_0$ face maps and $s_0$ degeneracy maps. The original $d_0$ maps induce a simplicial map $d_\bot : \Dec_\bot (X) \to X$, called the décalage map.
Similarly, the \emph{upper décalage} $\Dec_\top (X)$ is the simplicial space obtained by deleting $X_0$, all last face maps $d_\top$ and last degeneracy maps $s_\top$. The original $d_\top$ maps induce a simplicial map $d_\top : \Dec_\top X \to X$.

\begin{prop}[{\cite[Proposition 4.9]{GKT1}}]
    If $X$ is a decomposition space then $\Dec_\top X$ and $\Dec_\bot X$ are Segal spaces, and the décalage maps are culf.
\end{prop}

\subsection*{Bicomodules}

A \emph{bisimplicial groupoid} is an object of the functor category
\[
    \Fun{(\simplexcat\op \times \simplexcat\op, \Grpd)}.
\]
The horizontal face and degeneracy maps will be denoted by $d_k$ and $s_k$, and the vertical ones by $e_k$ and $t_k$.
A \emph{double Segal space} is a bisimplicial groupoid satisfying the Segal condition for each restriction $\simplexcat\op \times \{[n]\} \to \Grpd$ (the columns) and $\{[n]\} \times \simplexcat\op \to \Grpd$ (the rows).
An \emph{augmented bisimplicial groupoid} is an object of the functor category
\[
    \Fun{(\simplexcat\op_{+} \times \simplexcat\op_{+} \backslash \{(-1,-1)\}, \Grpd)},
\]
where $\simplexcat_{+}$ is the augmented simplex category of all finite ordinals and order-preserving maps.

\emph{Stability} was introduced in \cite{BOORS} for bisimplicial sets, and in \cite{Ca} for bisimplicial $\infty$-groupoids. In the case where every row and column of the bisimplicial groupoid is Segal, the stability condition has a simple expression \cite[Lemma 2.3.3]{Ca}: a double Segal space $B: \simplexcat\op \times \simplexcat\op \to \Grpd$ is stable if just these two squares are pullbacks
\begin{center}
  \begin{tikzcd}
    B_{0,0} 
        & B_{0,1}  \arrow[l, "d_0"'] \\
    B_{1,0} \arrow[u, "e_0"]  
        & B_{1,1}, \arrow[l, "d_0"] 
                  \arrow[u, "e_0"'] 
                  \arrow[ul, phantom, "\ulcorner", very near start]
  \end{tikzcd}
  \qquad
  \begin{tikzcd}
    B_{0,0} 
        & B_{0,1}  \arrow[l, "d_1"'] \\
    B_{1,0} \arrow[u, "e_1"]  
        & B_{1,1}. \arrow[l, "d_1"] 
                  \arrow[u, "e_1"'] 
                  \arrow[ul, phantom, "\ulcorner", very near start]
  \end{tikzcd}
\end{center}

A \emph{comodule configuration} is a culf map $v: C \to Y$ from a Segal space $C$ to a decomposition space $Y$.
A \emph{right pointed comodule configuration} is a comodule configuration $C \to Y$ such that the Segal space $C$ is augmented, and with new bottom sections $s_{-1}: C_{n-1} \to C_n$, see \cite[\S 4]{Ca}.
A right pointed comodule configuration $v: C \to Y$ is \emph{complete} if the new degeneracies $s_{-1}: C_{n-1} \to C_{n}$ are monomorphisms.
A \emph{right Möbius comodule configuration} is a complete right pointed comodule configuration $C \to Y$ such that the decomposition space $Y$ is Möbius and the augmented comodule is Möbius, that is
\begin{itemize}
    \item $C$ is locally finite: the groupoid $C_0$ is finite and both $s_{-1}$ and $d_0$ are finite maps;
    \item $C$ is of locally finite length: for all $a \in C_0$, the fibres of $d_0^{(n)}: \ora{C}_{n} \to C_0$ over $a$ are empty for $n$ sufficiently large, where $\ora{C}_{n}$ is the full subgroupoid of simplices with all principal edges nondegenerate \cite[\S 4.5]{Ca}.  
\end{itemize}
Similarly a \emph{left Möbius comodule configuration} is a complete left pointed comodule configuration $u: D \to X$ such that the decomposition space $X$ is Möbius and the augmented comodule is Möbius, using the extra top degeneracy map $t_{\top+1}$ and the top face map $e_{\top}$.

A \emph{bicomodule configuration} is an augmented stable double Segal space $B$ such that the augmentation maps are culf, and moreover $B_{\bullet, -1}$ and $B_{-1,\bullet}$ are decomposition spaces.
A \emph{Möbius bicomodule configuration} is a bicomodule configuration such that the comodule configuration $B_{0,\bullet} \to B_{-1,\bullet}$ is right pointed, the comodule configuration $B_{\bullet, 0} \to B_{\bullet, -1}$ is left pointed, and such that both left and right comodule configurations are Möbius.
The following picture represents this augmented bisimplicial groupoid, where $X := B_{\bullet, -1}$ and $Y := B_{-1,\bullet}$, and with the extra sections indicated with dotted arrows:
\begin{center}
    \begin{tikzcd}[sep=large]
         & Y_0 \ar[r, "s_0" description, near end] 
         \ar[d, bend left,shift left=1, dotted]
         & Y_1 \ar[l, shift left=1.5, "d_0", near end]
                    \ar[l, shift right=1.5, "d_1"', near end]
                    \ar[r, shift left=1.5, "s_1" description, near end] 
                    \ar[r, shift right=1.5, "s_0" description, near end]
         & Y_2 \ar[l, "d_1" description, near end]
                    \ar[l, shift left=3, "d_0", near end] 
                    \ar[l, shift right=3, "d_2"', near end] 
                    \ar[r, phantom, "\dots"]
                    & \phantom{} \\
       X_0 \ar[d, "t_0" description]
       \ar[r, bend right, dotted]
         & B_{0,0} \ar[d, "t_0" description]
                   \ar[u, "v"']
                   \ar[l, "u"']
                   \ar[r, "s_0" description] 
                   \ar[r, bend right,shift right=1, dotted]
                   \ar[d, bend left,shift left=2, dotted]
         & B_{0,1} \ar[d, "t_0" description]
                   \ar[u, "v"']
                   \ar[l, shift left=1.5, "d_0", near end]
                    \ar[l, shift right=1.5, "d_1"', near end]
                   \ar[r, shift left=1.5, "s_1" description, near end] 
                   \ar[r, shift right=1.5, "s_0" description, near end] 
                   \ar[r, bend right,shift right=2, dotted]
         & B_{0,2} \ar[u, "v"']
                   \ar[d, "t_0" description]
                   \ar[l, "d_1" description, near end]
                   \ar[l, shift left=3, "d_0", near end] 
                   \ar[l, shift right=3, "d_2"', near end]
                   \ar[r, phantom, "\dots"]
                    & \phantom{}\\
       X_1 \ar[u, shift left=1.5, "e_0", near end]
                \ar[u, shift right=1.5, "e_1"', near end]
                \ar[d, phantom, "\vdots"]
         & B_{1,0} \ar[u, shift left=1.5, "e_0", near end]
                   \ar[u, shift right=1.5, "e_1"', near end]
                   \ar[l, "u"'] 
                   \ar[r, "s_0" description, near end] 
                    \ar[d, phantom, "\vdots"]
         & B_{1,1} \ar[u, shift left=1.5, "e_0", near end]
                   \ar[u, shift right=1.5, "e_1"', near end]
                   \ar[l, shift left=1.5, "d_0", near end]
                    \ar[l, shift right=1.5, "d_1"', near end]
                   \ar[r, shift left=1.5, "s_1" description, near end] 
                   \ar[r, shift right=1.5, "s_0" description, near end] 
                    \ar[d, phantom, "\vdots"]
         & B_{1,2}   
                   \ar[u, shift left=1.5, "e_0", near end]
                   \ar[u, shift right=1.5, "e_1"', near end]
                   \ar[l, "d_1" description, near end]
                   \ar[l, shift left=3, "d_0", near end] 
                   \ar[l, shift right=3, "d_2"', near end]
                   \ar[r, phantom, "\dots"]
                    \ar[d, phantom, "\vdots"]
                    & \phantom{}\\
       \phantom{}
          & \phantom{}
          &\phantom{}
          & \phantom{}
    \end{tikzcd}
\end{center}

\begin{thm}[Generalised Rota formula {\cite[Theorem 4.5.2]{Ca}}]\label{genRota}
    Given a Möbius bicomodule configuration $B$ with $X := B_{\bullet, -1}$ and $Y:=B_{-1, \bullet}$, we have
\[
    |\mu^{X}| \star_l |\delta^R| = |\delta^{L}| \star_r |\mu^Y|,
\]
where $\star_l$ and $\star_r$ are dual to the comodule structures,
 $\delta^R$ is the linear functor given by the span
$    B_{0,0} \xleftarrow{} Y_{0} \xrightarrow{} 1$,
and $\delta^{L}$ is the linear functor given by the span
$    B_{0,0} \xleftarrow{} X_{0} \xrightarrow{} 1$.
\end{thm}
Any adjunction of categories induces a bicomodule configuration (which is Möbius if the two categories are), and the formula reduces to the Rota formula in the case of a Galois connection between locally finite posets.

\subsection*{Layered finite posets and layered finite sets}

We refer to \cite{GKT:restrict} for the following material.
An \emph{$n$-layering} of a finite poset $P$ is a monotone map $l: P \to \underline{n}$, where $\underline{n} = \{1, \dots, n \}$ are the objects of the skeleton of the category of finite ordered sets (possibly empty) and monotone maps. The fibres $P_i = l^{-1}(i)$, $i \in \underline{n}$ are called layers, and can be empty.
The objects of the groupoid $\ds{C}_n$ of $n$-layered finite posets are monotone maps $l: P \to \underline{n}$ and the morphisms are triangles
\begin{center}
    \begin{tikzcd}[column sep=small]
        P \arrow[rr, "\eq"] 
           \arrow[dr, ""'] & & P' \arrow[dl, ""]\\
        & \underline{n}, & 
    \end{tikzcd}
\end{center}
where $P \to P'$ is a monotone bijection.
They assemble into a simplicial groupoid $\ds C$. 
The face maps are given by joining layers, or deleting an outer layer for the top and bottom maps. The degeneracy maps are given by inserting empty layers. 
\begin{prop}[{\cite[Proposition 6.12, Lemma 6.13]{GKT:restrict}}]
    The simplicial groupoid $\ds C$ of layered finite posets is a decomposition space (but not a Segal space), and is complete, locally finite, locally discrete, and of locally finite length.
\end{prop}
The incidence coalgebra of $\ds{C}$ has comultiplication given by the span
\[
    \ds{C}_1 \xla{d_1} \ds{C}_2 \xra{(d_2, d_0)} \ds{C}_1 \times \ds{C}_1,
\]
where $d_1$ joins the two layers, and $d_2$ and $d_0$ return the two layers.
The comultiplication of a poset is thus obtained by summing over admissible cuts (a $2$-layering of the poset) and taking tensor product of the two layers.

\bigskip

Similarly, let $\mathbf I_n$ denote the groupoid of all layerings of finite sets. Again these groupoids assemble into a simplicial groupoid, denoted $\ds{I}$.

\begin{prop}[{\cite[Proposition 4.3, Lemma 4.4]{GKT:restrict}}]
    The simplicial groupoid $\ds{I}$ is a Segal space, and hence a decomposition space, which is complete, locally finite, locally discrete, and of locally finite length.
\end{prop}

The simplicial groupoid $\ds{C}$ is the decomposition space corresponding to the terminal directed restriction species, finite posets and convex maps, while $\ds{I}$ is the decomposition space corresponding to the terminal restriction species, finite sets and injections.
The incidence coalgebra of $\ds{I}$ is the binomial coalgebra \cite[\S 2.4]{GKT:restrict} with well-known Möbius function $(-1)^n$ for a set with $n$ elements.

\section{Bisimplicial groupoids, abacus maps, and bicomodule configurations}\label{sec:bicomodule}

We want to define a bicomodule configuration interpolating between the decomposition space $\ds{C}$ of finite posets and the decomposition space $\ds{I}$ of finite sets, in order to relate the Möbius functions of the incidence algebras of these decomposition spaces. As explained in the introduction, we shall achieve this by modifying the box product $\ds{I} \Boxprod \Dec_\bot \ds{C}$, and we introduce the notion of abacus map for a bisimplicial groupoid for this purpose. The modification is necessary in order to be able to define an extra vertical degeneracy map, in turned required to establish the Möbius property.



\subsection*{Abacus maps}
\label{sub:abacus}

Let $B$ be a bisimplicial groupoid with horizontal face and degeneracy maps denoted by $d_k$ and $s_k$, and vertical face and degeneracy maps denoted by $e_k$ and $t_k$. A family $f$ of maps $f_{i,j}: B_{i+1,j} \to B_{i,j+1}$ is called an \emph{abacus map} if 
\begin{itemize}
    \item for all $i$, the map $f_{i,\bullet}: B_{i+1,\bullet} \to \Dec_\bot (B_{i,\bullet})$ is simplicial (between rows), 
    \item for all $j$, the map $f_{\bullet,j}: \Dec_\top(B_{\bullet,j}) \to B_{\bullet,j+1}$ is simplicial (between columns) except for the top face map, 
    \item $d_\bot f_{i,j} t_\top = \id$, where $d_\bot$ is the horizontal bottom face map, $t_\top$ is the vertical top degeneracy map.
\end{itemize}

\begin{thm}\label{bisimplicialalt}
    Let $B$ be a bisimplicial groupoid, and $f$ an abacus map.
    Define new vertical top degeneracy maps $\tilde{e}_\top := d_\bot f_{i,j}$. Then the groupoids $B_{i,j}$ with the new $\tilde{e}_\top$ form a bisimplicial groupoid, denoted $\tilde{B}$ (for which $f$ is still an abacus map). 
\end{thm}

\begin{proof}
Firstly, let us prove that the new $\tilde{e}_\top$ is simplicial between rows:
since the map $f_{i, \bullet}$ is simplicial, and by the face-map identities for the rows, we have
\[
\tilde{e}_\top d_k = d_\bot f_{i,j} d_k = d_\bot d_{k+1} f_{i,j+1} = d_{k} d_\bot f_{i,j+1} = d_k \tilde{e}_\top,
\]
and similarly for $s_k$.

Secondly, let us check that the columns are simplicial.
The simplicial identities involving $\tilde{e}_\top = \tilde{e}_i: B_{i,j} \to B_{i-1, j}$ are:
\begin{enumerate}
    \item $t_k \tilde{e}_\top = \tilde{e}_\top t_k$, for $k < i$;
    \item $e_k \tilde{e}_\top = \tilde{e}_\top e_k$, for $k < i$;
    \item $\tilde{e}_\top t_\top = \id$.
\end{enumerate}
To verify the first two identities, we use the commutativity of horizontal maps against vertical maps, and that $f_{\bullet, j}$ is simplicial except for the top maps.
The third identity is exactly the last condition in the definition of an abacus map.
\end{proof}

We say an abacus map is \emph{perfect} if $f_{i,j} e_{\top-1} = d_\bot f_{i,j+1} f_{i+1,j}$ and $e_\top = d_\bot f_{i,j}$ for all $i,j$. We get the following proposition, whose proof is straightforward.

\begin{prop}\label{simplicialandidempotent}
    Let $B$ be a bisimplicial groupoid and $f$ a perfect abacus map.
    Then the map $f_{\bullet,j}: \Dec_\top(\tilde{B}_{\bullet,j}) \to \tilde{B}_{\bullet,j+1}$ simplicial, and the construction of Theorem~\ref{bisimplicialalt} is idempotent.
\end{prop}

Given an augmented bisimplicial groupoid $B$, we say an abacus map is \emph{left augmented} if there are maps 
$f_{i,-1}: B_{i+1,-1} \to B_{i,0}$ 
such that
$f_{i,\bullet}: B_{i+1,\bullet} \to \Dec_\bot (B_{i,\bullet})$
is augmented simplicial for all $i\ge 0$, that is the following diagram commutes
\begin{center}
    \begin{tikzcd}[column sep=scriptsize]
       & B_{i,0} & B_{i,1} \ar[l, "d_1"'] \\
       B_{i+1,-1} \ar[ur, "f_{i,-1}"]& B_{i+1,0} \ar[l, "u"]\ar[ur, "f_{i,0}"'] &
    \end{tikzcd}
\end{center}
and moreover $u f_{i,-1} = e_\top$:
\begin{center}
    \begin{tikzcd}
       B_{i,-1 } & B_{i,0} \ar[l, "u"'] \\
       B_{i+1,-1} \ar[u, "e_\top"] \ar[ur, "f_{i,-1}"']& 
    \end{tikzcd}
\end{center}

We say an abacus map is \emph{right augmented} if there are maps
$f_{-1,j}: B_{0,j} \to B_{-1,j+1}$ such that $f_{\bullet,j}: \Dec_\top(B_{\bullet,j}) \to B_{\bullet,j+1}$ is augmented simplicial except for the top face map for all $j \ge 0$, that is the following diagram commutes
\begin{center}
  \begin{tikzcd}[column sep=scriptsize]
       & B_{-1,j}  \\
       B_{0,j} \ar[ur, "f_{-1,j}"] & B_{0,j+1} \ar[u, "v"'] \\
       B_{1,j} \ar[u, "e_0"] \ar[ur, "f_{0,j}"'] &
  \end{tikzcd}
\end{center}
and moreover $v = d_0 f_{-1,0}$:
\begin{center}   
  \begin{tikzcd}
       B_{-1,j} & B_{-1,j+1} \ar[l, "d_0"']  \\
       B_{0,j} \ar[u, "v"] \ar[ur, "f_{-1,j}"'] &   \\
  \end{tikzcd}
\end{center}

We say an abacus map is \emph{augmented} if it is left and right augmented.

\begin{lem}\label{usimplicial}
    Given an augmented bisimplicial groupoid $B$ and a left augmented abacus map. Then the augmentation map $u: \tilde{B}_{\bullet, 0} \to \tilde{B}_{\bullet, -1}$ is simplicial.
\end{lem}

\begin{proof}
    We only need to check the commutativity with the new top face map, which is a direct verification:
  $
    u \tilde e_\top = u d_0 f_{i,0} = u d_1 f_{i,0} =  u f_{i,-1} u = e_\top u$.
\end{proof}

\begin{lem}\label{vaugmentation}
    Given an augmented bisimplicial groupoid $B$ and a right augmented abacus map. Then $v : \tilde B_{0,\bullet} \to \tilde B_{-1, \bullet}$ is an augmentation map.
\end{lem}

\begin{proof}
    We only need to verify it coequalises $e_0$ and $\tilde e_1$.
    We have
    $v e_0 = d_0 f_{-1,j} e_0 = d_0 v f_{0,j} = v d_0 f_{0,j} = v \tilde e_1$.
\end{proof}

\begin{rmk}
A augmented perfect abacus map produces an example of a 
cocartesian nerve in the sense of \cite{Ca}, 
that is a map
$N_{\text{cocart}}: \Cat_{\infty/\Delta^1}^{\text{cocart}} \to \Fun((\overline{\simplexcat_{/\Delta^1}})\op, \infGrpd)$ where
$N_{\text{cocart}}(p)_{\overline{i,j}}$ is a mapping space preserving cocartesian arrows, and $\overline{\simplexcat_{/\Delta^1}}$ is a category of shape like $\simplexcat_{/\Delta^1}$, but with extra diagonal maps.
This ensures $\Grpd_{/B_{0,0}}$ is pointed as a right comodule over $\Grpd_{/B_{1,-1}}$.
\end{rmk}

\begin{ex}[Bisimplicial groupoid associated to a functor]
Given a functor $F:X \to Y$ between categories, we consider the groupoid $B_{i,j}$ whose objects consist of $(i{+}j{+}1)$-tuples of composable morphisms such that the $i$ first morphisms are given as images of morphisms in $X$.
An object in $B_{i,j}$ can be pictured as follows:
\begin{center}
      \begin{tikzcd}[column sep=tiny]
        Fx \ar[rr, "Ff"] & & Fx' \ar[rr, ""] & & \dots \ar[rrr] & & & Fx^{(i)} \ar[dllllll] \\
       & y^{(j)} \ar[rr] & & y^{(j-1)} \ar[rr]  && \dots \ar[rr] & & y. 
      \end{tikzcd}
\end{center}

The groupoids $B_{i,j}$ assemble into a (augmented) bisimplicial groupoid, where horizontal face and degeneracy maps are given by face and degeneracy maps of the nerve of $Y$, and similarly, vertical face and degeneracy maps are given by face and degeneracy maps of the nerve of $X$. This is moreover a bicomodule configuration.
There is an (augmented) perfect abacus map sending
\begin{center}
      \begin{tikzcd}[column sep=tiny]
        Fx \ar[rr, ""]  & & Fx' \ar[dl] \\
       & y & 
      \end{tikzcd}
\qquad to \qquad
      \begin{tikzcd}[column sep=tiny]
       & Fx \ar[dl, ""] & \\
       Fx' \ar[rr] & & y.
      \end{tikzcd}  
\end{center}
\end{ex}

\subsection*{Bicomodule configurations}
Given two simplicial groupoids $X$ and $Y$, their {\em box product} \cite{Joyal-Tierney} is the bisimplicial groupoid $X \Boxprod Y$ given by the groupoids $X_i \times Y_j$, with horizontal and vertical face and degeneracy maps induced by those of $X$ and $Y$.
  We shall be concerned rather with the box product
 \[
  B:= X \Boxprod \Dec_\bot Y ,
  \]
and use the following notation:
    the horizontal maps $\underline{d}_k: B_{i,j} \to B_{i, j-1}$ and $\underline{s}_k: B_{i,j} \to B_{i, j+1}$ are given by:
    \[
        \underline{d}_k = \id_{X_i} \times\,d_{k+1}, \qquad
        \underline{s}_k = \id_{X_i} \times\,s_{k+1}, \qquad
        0 \le k \le j,
    \]
     where $d_k$ and $s_k$ are the face and degeneracy maps of $Y$.
    The vertical maps $\underline{e}_k : B_{i,j} \to B_{i-1, j}$ and  $\underline{t}_k : B_{i,j} \to B_{i+1, j}$ are given by:
        \[
        \underline{e}_k = e_{k} \times \id_{Y_{j+1}}, \qquad 
        \underline{t}_k = t_{k} \times \id_{Y_{j+1}}, \qquad
        0 \le k \le i,
    \]
     where $e_k$ and $t_k$ are the face and degeneracy maps of $X$.
Note that since the second factor is given by décalage, there is also an extra bottom degeneracy map given by $\underline{s}_{-1} = \id_{X_i} \times s_{0}$.
There are also augmentation maps $u: B_{\bullet, 0} \to X_\bullet \times Y_0$ given by $\id_X \times d_1$, and $v: B_{0, \bullet} \to X_0 \times Y$ given by the décalage map.

The following trivial lemma will be invoked several times.
\begin{lem}\label{pbkproduct}
The following square of groupoids is a pullback:
\begin{center}
    \begin{tikzcd}
        X \times Y \ar[r, "f \times \id_Y"] 
           \ar[d, "\id_X \times g"'] & X' \times Y \ar[d, "\id_{X'} \times g"]\\
        X \times Y' \ar[r, "f \times \id_{Y'}"'] & X' \times Y'.
    \end{tikzcd}
\end{center}
\end{lem}

\begin{prop}\label{bicomoduleconf}
   Suppose $f$ is an augmented abacus map for $B := X \Boxprod \Dec_\bot Y$, such that $f_{\bullet,j}$ is a right fibration (that is cartesian on bottom face maps $\underline{e}_\bot$), and $f_{i,\bullet}$ is a left fibration (that is cartesian on top face maps $\underline{d}_\top$).
    Then the modified augmented bisimplicial groupoid $\tilde{B}$ with the new $\tilde{\underline{e}}_\top$ obtained from the abacus map as in Theorem~\ref{bisimplicialalt} form an augmented bisimplicial groupoid which is double Segal, stable, and such that the augmentation maps are culf.
\end{prop}

The bisimpliciality follows from Theorem~\ref{bisimplicialalt}. We split the rest of the proof into the following Lemmas~\ref{doubleSegal}-\ref{stable}.

\begin{lem}\label{doubleSegal}
    Suppose $f_{\bullet,j}$ is a right fibration, that is cartesian on bottom face maps $\underline{e}_\bot$.
    Then for every $i \ge 0$, the simplicial groupoid $\tilde{B}_{i,\bullet}$ is Segal, and for every $j \ge 0$ the simplicial groupoid $\tilde{B}_{\bullet,j}$ is Segal.
\end{lem}

\begin{proof}
The simplicial groupoid $\tilde{B}_{i,\bullet}$ is Segal since it is the product with the groupoid $X_i$ of the décalage of the decomposition space $Y$.
The simplicial groupoid $\tilde{B}_{\bullet,j}$ is Segal if, for all $j \ge 0$, and $n \ge 1$, the following square is a pullback
\begin{center}
    \begin{tikzcd}
        \tilde{B}_{n+1,j} \ar[d, "\underline{e}_0"'] \ar[r, "\tilde{\underline{e}}_{n+1}"] & \tilde{B}_{n,j} \ar[d, "\underline{e}_0"] \\
        \tilde{B}_{n,j} \ar[r, "\tilde{\underline{e}}_{n}"']  & \tilde{B}_{n-1,j}.
    \end{tikzcd}
\end{center}
It is the outer square in the following diagram:
\begin{center}
    \begin{tikzcd}
        \tilde{B}_{n+1,j} \ar[d, "\underline{e}_0"'] \ar[r, "f_{n,j}"] & \tilde{B}_{n,j+1} \ar[d, "\underline{e}_0"']  \ar[r, "\underline{d}_0"] & \tilde{B}_{n,j} \ar[d, "\underline{e}_0"] \\
        \tilde{B}_{n,j} \ar[r, "f_{n-1,j}"'] & \tilde{B}_{n-1,j+1} \ar[r, "\underline{d}_0"']  & \tilde{B}_{n-1,j}.
    \end{tikzcd}
\end{center}
The right-hand square is a pullback by Lemma~\ref{pbkproduct}.
The left-hand square is a pullback because $f_{\bullet,j}$ is a right fibration.
\end{proof}

\begin{lem}\label{culf}
    The augmentation maps $u: \tilde{B}_{\bullet,0} \to X \times Y_0$ and
$v: \tilde{B}_{0, \bullet} \to X_0 \times Y$ are culf.
\end{lem}

\begin{proof}
Note that $u$ is a simplicial map by Lemma~\ref{usimplicial} and that $v$ is an augmentation map by Lemma~\ref{vaugmentation}.
It is enough to check the two following squares are pullbacks \cite[Lemma 4.3]{GKT1}:
\begin{center}
    \begin{tikzcd}
        X_{1} \times Y_0  & \tilde{B}_{1,0} \ar[l, "u"'] \\
        X_2 \times Y_0 \ar[u, "e_1 \times \id_{Y_0}"] & \tilde{B}_{2,0}, \ar[u, "\underline{e}_1"'] \ar[l, "u"]
    \end{tikzcd}
    \qquad
    \begin{tikzcd}
       X_0 \times  Y_{1}  & X_0 \times Y_{2} \ar[l, "d_1"'] \\
        \tilde{B}_{0,1} \ar[u, "v"] & \tilde{B}_{0,2}. \ar[u, "v"'] \ar[l, "\underline{d}_1"]
    \end{tikzcd}
\end{center}
By definition, the left-hand one is the following square
\begin{center}
    \begin{tikzcd}[sep=large]
        X_{1} \times Y_0 & X_1 \times Y_1 \ar[l, "\id_{X_1} \times d_1"'] \\
        X_2 \times Y_0 \ar[u, "e_1 \times \id_{Y_0}"] & X_2 \times Y_1, \ar[u, "e_1 \times \id_{Y_1}"'] \ar[l, "\id_{X_2} \times d_1"]
    \end{tikzcd}
\end{center}
which is a pullback by Lemma~\ref{pbkproduct}.
The right-hand square is a pullback since the décalage map of a decomposition space is culf.
\end{proof}

\begin{lem}\label{stable}
    Suppose $f_{i,\bullet}$ is a left fibration, that is cartesian on top face maps $d_\top$.
    Then the bisimplicial groupoid $\tilde{B}$ is stable.
\end{lem}

\begin{proof}
Since the bisimplicial groupoid $\tilde{B}$ is a double Segal space by Lemma \ref{doubleSegal}, the stability can be established checking only the two following squares are pullbacks \cite[Lemma 2.3.3]{Ca}: 
\begin{center}
    \begin{tikzcd}
        \tilde{B}_{0,0} & \tilde{B}_{0,1} \ar[l, "\underline{d}_0"'] \\
        \tilde{B}_{1,0} \ar[u, "\underline{e}_0"]  & \tilde{B}_{1,1}, \ar[u, "\underline{e}_0"']  \ar[l, "\underline{d}_0"]
    \end{tikzcd}
    \qquad
        \begin{tikzcd}
        \tilde{B}_{0,0} & \tilde{B}_{0,1} \ar[l, "\underline{d}_\top"'] \\
        \tilde{B}_{1,0} \ar[u, "\tilde{\underline{e}}_\top"]  & \tilde{B}_{1,1} \ar[u, "\tilde{\underline{e}}_\top"']  \ar[l, "\underline{d}_\top"].
    \end{tikzcd}
\end{center}
The first square is a pullback by Lemma~\ref{pbkproduct}.
The second square is the outer rectangle of the following diagram:\begin{center}
    \begin{tikzcd}
        \tilde{B}_{0,0} & \tilde{B}_{0,1} \ar[l, "\underline{d}_\top"'] \\
        \tilde{B}_{0,1} \ar[u, "\underline{d}_0"] & \tilde{B}_{0,2} \ar[u, "\underline{d}_0"'] \ar[l, "\underline{d}_\top"']\\
        \tilde{B}_{1,0} \ar[u, "f_{0,0}"]  & \tilde{B}_{1,1} \ar[u, "f_{0,1}"']  \ar[l, "\underline{d}_\top"].
    \end{tikzcd}
\end{center}
The top square is a pullback since $Y$ is a decomposition space. 
The bottom one is a pullback because $f_{i,\bullet}$ is a left fibration.
\end{proof}


\subsection*{Layered sets and posets}\label{layeredsetsandposets}
We now specialise to the situation where $X = \ds{I}$ the decomposition space of layered finite sets and $Y = \ds{C}$ the decomposition space of layered finite posets. 


\begin{lem}\label{abacusposets}
    The simplicial groupoid $B:= \ds{I} \Boxprod \Dec_\bot \ds{C}$ has an augmented abacus map with $f_{i,j}: B_{i+1,j} \to B_{i, j+1}$ given by moving the last layer of the set into a new first layer of the poset.
\end{lem}

\begin{proof}
Remark that the augmentation column is $\ds I$ and the augmentation row is $\ds C$.
The groupoid $B_{i,j}$ consists of pairs of layerings $(S \to \underline{i},P \to \underline{j{+}1})$ where $S$ is a finite set, and $P$ is a finite poset.
The map $f_{i,j}$ sends $(S \xra{a} \underline{i},P \xra{b} \underline{j{+}1})$ to $(a^{-1}(\underline{i{-}1}) \xra{} \underline{i{-}1},(a^{-1}(i) + P) \xra{} \underline{1{+}j{+}1})$.
The following picture represents the map $f_{2,1}$ sending an element in the groupoid $B_{3,1}$ to an element in the groupoid $B_{2,2}$, where layers are numbered from bottom to top.
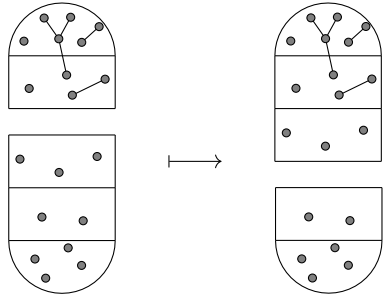
\begin{center}
    \begin{tikzpicture}[x=0.7cm,y=0.7cm]
\draw [shift={(0,0)}]  plot[domain=0:3.141592653589793,variable=\t]({1*1*cos(\t r)+0*1*sin(\t r)},{0*1*cos(\t r)+1*1*sin(\t r)});
\draw  (-1,0)-- (1,0);
\draw  (-0.33507039454439097,0.7155329229551461)-- (-0.060062480287099776,0.3216026673974046);
\draw  (-0.060062480287099776,0.3216026673974046)-- (0.16291690965124442,0.7303982156177022);
\draw  (0.37103100692703234,0.2547088504159014)-- (0.6980674455032706,0.5520147036670269);
\draw  (-1,0)-- (-1,-1);
\draw  (1,-1)-- (-1,-1);
\draw  (1,-1)-- (1,0);
\draw  (0.08733605190851508,-0.36490821962635805)-- (-0.060062480287099776,0.3216026673974046);
\draw  (0.8102764549536944,-0.4430639388744853)-- (0.19480016587469037,-0.7459173509609783);
\draw  (-1,-1.5)-- (1,-1.5);
\draw  (1,-1.5)-- (1,-2.5);
\draw  (1,-2.5)-- (1,-3.5);
\draw  (1,-3.5)-- (-1,-3.5);
\draw  (-1,-3.5)-- (-1,-2.5);
\draw  (-1,-2.5)-- (-1,-1.5);
\draw  (-1,-2.5)-- (1,-2.5);
\draw [shift={(0,-3.5)}]  plot[domain=3.141592653589793:6.283185307179586,variable=\t]({1*1*cos(\t r)+0*1*sin(\t r)},{0*1*cos(\t r)+1*1*sin(\t r)});
\draw [|->] (2,-2) -- (3,-2);
\draw  (4.677222517052907,0.7185775139350105)-- (4.952230431310198,0.3246472583772689);
\draw  (4.952230431310198,0.3246472583772689)-- (5.175209821248543,0.7334428065975666);
\draw  (5.38332391852433,0.25775344139576567)-- (5.710360357100568,0.5550592946468911);
\draw  (5.099628963505813,-0.3618636286464939)-- (4.952230431310198,0.3246472583772689);
\draw  (5.822569366550993,-0.4400193478946205)-- (5.207093077471988,-0.7428727599811142);
\draw  (6.012292911597298,-2.496955409020136)-- (6.012292911597298,-3.4969554090201367);
\draw  (6.012292911597298,-3.4969554090201367)-- (4.0122929115972985,-3.4969554090201367);
\draw  (4.0122929115972985,-3.4969554090201367)-- (4.0122929115972985,-2.496955409020136);
\draw  (4.0122929115972985,-2.496955409020136)-- (6.012292911597298,-2.496955409020136);
\draw [shift={(5.012292911597298,-3.4969554090201367)}]  plot[domain=3.141592653589793:6.283185307179586,variable=\t]({1*1*cos(\t r)+0*1*sin(\t r)},{0*1*cos(\t r)+1*1*sin(\t r)});
\draw  (4,-1)-- (4,-2);
\draw  (6,-2)-- (4,-2);
\draw  (6,-2)-- (6,-1);
\draw  (6,-1)-- (4,-1);
\draw  (4,-1)-- (4,0);
\draw  (4,0)-- (6,0);
\draw  (6,0)-- (6,-1);
\draw [shift={(5,0)}]  plot[domain=0:3.141592653589793,variable=\t]({1*1*cos(\t r)+0*1*sin(\t r)},{0*1*cos(\t r)+1*1*sin(\t r)});
\draw [fill=gray] (-0.33507039454439097,0.7155329229551461) circle (1.5pt);
\draw [fill=gray] (-0.060062480287099776,0.3216026673974046) circle (1.5pt);
\draw [fill=gray] (0.16291690965124442,0.7303982156177022) circle (1.5pt);
\draw [fill=gray] (0.37103100692703234,0.2547088504159014) circle (1.5pt);
\draw [fill=gray] (0.6980674455032706,0.5520147036670269) circle (1.5pt);
\draw [fill=gray] (-0.714135357439576,0.27700678940973567) circle (1.5pt);
\draw [fill=gray] (-0.6160654213246324,-0.6189143071827722) circle (1.5pt);
\draw [fill=gray] (0.08733605190851508,-0.36490821962635805) circle (1.5pt);
\draw [fill=gray] (0.8102764549536944,-0.4430639388744853) circle (1.5pt);
\draw [fill=gray] (0.19480016587469037,-0.7459173509609783) circle (1.5pt);
\draw [fill=gray] (-0.7893105481025529,-1.9579677922599437) circle (1.5pt);
\draw [fill=gray] (-0.05451315709960766,-2.2096107343842393) circle (1.5pt);
\draw [fill=gray] (0.6601527985333938,-1.9076392038350845) circle (1.5pt);
\draw [fill=gray] (0.39844413872412565,-3.125591043716676) circle (1.5pt);
\draw [fill=gray] (-0.37661612301870695,-3.0551310199218733) circle (1.5pt);
\draw [fill=gray] (-0.507470452923341,-3.850322717034648) circle (1.5pt);
\draw [fill=gray] (-0.30615609922390397,-4.212688553693634) circle (1.5pt);
\draw [fill=gray] (0.11660404354491381,-3.6389426456502396) circle (1.5pt);
\draw [fill=gray] (0.3682469856692101,-3.9711113292543097) circle (1.5pt);
\draw [fill=gray] (4.677222517052907,0.7185775139350105) circle (1.5pt);
\draw [fill=gray] (4.952230431310198,0.3246472583772689) circle (1.5pt);
\draw [fill=gray] (5.175209821248543,0.7334428065975666) circle (1.5pt);
\draw [fill=gray] (5.38332391852433,0.25775344139576567) circle (1.5pt);
\draw [fill=gray] (5.710360357100568,0.5550592946468911) circle (1.5pt);
\draw [fill=gray] (4.298157554157722,0.2800513803895999) circle (1.5pt);
\draw [fill=gray] (4.396227490272666,-0.6158697162029081) circle (1.5pt);
\draw [fill=gray] (5.099628963505813,-0.3618636286464939) circle (1.5pt);
\draw [fill=gray] (5.822569366550993,-0.4400193478946205) circle (1.5pt);
\draw [fill=gray] (5.207093077471988,-0.7428727599811142) circle (1.5pt);
\draw [fill=gray] (4.215373413727493,-1.4603414664086674) circle (1.5pt);
\draw [fill=gray] (4.950170804730438,-1.7119844085329605) circle (1.5pt);
\draw [fill=gray] (5.664836760363439,-1.410012877983808) circle (1.5pt);
\draw [fill=gray] (5.410737050321424,-3.1225464527368114) circle (1.5pt);
\draw [fill=gray] (4.635676788578591,-3.052086428942009) circle (1.5pt);
\draw [fill=gray] (4.504822458673957,-3.8472781260547846) circle (1.5pt);
\draw [fill=gray] (4.706136812373394,-4.209643962713772) circle (1.5pt);
\draw [fill=gray] (5.128896955142212,-3.635898054670376) circle (1.5pt);
\draw [fill=gray] (5.380539897266508,-3.968066738274448) circle (1.5pt);
\end{tikzpicture}
\end{center}
The verification of the abacus map axioms is straightforward.
\end{proof}

\begin{lem}\label{abacusposetsfib}
    The map $f_{\bullet,j}$ is a right fibration and the map $f_{i,\bullet}$ is a left fibration.
\end{lem}

\begin{proof}
        The map $f_{\bullet,j}$ is a right fibration: for each $(S,P) \in B_{n,j}$, the fibres of $e_0: B_{n+1,j} \to B_{n,j}$ along $(S,P)$ and $e_0: B_{n,j+1} \to B_{n-1,j+1}$ along $f_{n-1,j}(S,P)$ consist both of triples $(A, P' \xra{\alpha} P, S' \xra{\beta} S)$, where $A$ is a finite set and $\alpha$ and $\beta$ are monotone bijections:
\begin{center}
    \begin{tikzcd}
        F_{(S,P)} \ar[r] \ar[d] & B_{n+1,j} \ar[d, "e_0"] \ar[r, "f_{n,j}"] & B_{n,j+1} \ar[d, "e_0"] \\
        1 \ar[r, "\name{(S,P)}"'] & B_{n,j} \ar[r, "f_{n-1,j}"'] & B_{n-1,j+1}.
    \end{tikzcd}
\end{center}
 The map $f_{i,\bullet}$ is a left fibration: for each $(S,P) \in B_{i+1,j}$, the fibres of $d_\top : B_{i+1,j+1} \to B_{i+1,j}$ along $(S,P)$, and of $d_\top : B_{i,j+2} \to B_{i,j+1}$ along $f_{i, j+1}(S,P)$ are equivalent since they consist both of pairs $(S' \to \underline{i{+}1}, P' \xra{b} \underline{j{+}2})$, such that $S'\eq S$, and $b^{-1}(\{0, \dots, j+1\}) \eq P$.
\end{proof}

    The main object of interest will be the augmented bisimplicial groupoid obtained by applying Theorem~\ref{bisimplicialalt} to $B$.  We denote it $\ds{B}$.  By unpacking the general construction we get the following explicit description: the groupoid $\ds{B}_{i,j}$ consists of pairs of layerings $(S \to \underline{i},P \to \underline{j{+}1})$ where $S$ is a finite set, and $P$ is a finite poset.
For example, $\ds{B}_{0,0}$ is the groupoid of $1$-layered finite posets.
The horizontal face maps (taking place only on the $(j{+}1)$-layered finite poset part) are given by:
\begin{itemize}
    \item $d_k:\ds{B}_{i,j} \to \ds{B}_{i,j-1}$ joins the layers $(k{+}1)$ and $(k{+}2)$ of the poset, for all $j >0$ and $0 \le k \le j-1$;
    \item  $d_{\top} = d_{j}: \ds{B}_{i,j} \to \ds{B}_{i,j-1}$ deletes the last layer.
\end{itemize}
Horizontal degeneracy maps are given by inserting empty layers:
$s_k: \ds{B}_{i,j} \to \ds{B}_{i, j+1}$ inserts an empty $(k{+}2)$nd layer in the poset, for all $j\ge 0$ and $0\le k \le j$.

The vertical face maps are given by:
\begin{itemize}
    \item $e_\bot = e_0: \ds{B}_{i,j} \to \ds{B}_{i-1, j}$ deletes the first layer of the set, for all $i > 0$;
    \item $e_k: \ds{B}_{i,j} \to \ds{B}_{i-1, j}$ joins the layers $k$ and $k{+}1$ of the set, for all $0 < k < i$.
\end{itemize}
According to the modification, the top vertical face map is given by $e_\top = d_0 \circ f_{i,j}$.
\begin{itemize}
    \item $e_\top = e_i: \ds{B}_{i,j} \to \ds{B}_{i-1, j}$ joins the last layer of the set and the first layer of the poset into the first layer of the poset.
\end{itemize}
In this way, the top vertical map keeps some information about the last layer of the set, instead of just throwing it away.
Vertical degeneracy maps are given by inserting empty layers:  $t_k: \ds{B}_{i,j} \to \ds{B}_{i+1, j}$ inserts an empty $(k{+}1)$st layer to the set, for all $0 \le k \le i$.
The augmentation maps are  
$u: \ds{B}_{i,0} \to \ds{I}_i$ deleting the whole $1$-layered poset and
$v: \ds{B}_{0, j} \to \ds{C}_j$ deleting the first layer of the poset.
It should be noted that the row $\ds{B}_{0,\bullet}$ is the lower décalage of $\ds{C}$, that $v$ is the décalage map given by the original $d_0$, and that $u$ is the augmentation map that décalage always have.

\begin{prop}
    With augmentations maps $u$ and $v$, the simplicial groupoid $\ds{B}$ is a bicomodule configuration.
\end{prop}

\begin{proof}
 The result follows from Lemmas~\ref{abacusposets}-\ref{abacusposetsfib} and Proposition~\ref{bicomoduleconf}.
\end{proof}

\section{Möbius functions}
\subsection*{Möbius function of the decomposition space of finite posets}\label{sec:mobiusfunctions}

The augmented bisimplicial groupoid $\ds{B}$ of layered sets and posets is a bicomodule configuration between $\ds{I}$ and $\ds{C}$. By Theorem~2.4.1 of \cite{Ca}, the spans 
\[\ds{B}_{0,0}  \xleftarrow{e_1} \ds{B}_{1,0} 
    \xrightarrow{(u,e_0)} \ds{I}_{1} \times \ds{B}_{0,0}
\]
and
\[\ds{B}_{0,0}  \xleftarrow{d_0} \ds{B}_{0,1} 
    \xrightarrow{(d_1,v)} \ds{B}_{0,0} \times \ds{C}_{1}
\]
induce on $\Grpd_{/\ds{B}_{0,0}}$ the structure of a bicomodule over $\Grpd_{/\ds{I}_1}$ and $\Grpd_{/\ds{C}_1}$.

In order to be in position to apply the generalised Rota formula from Theorem~\ref{genRota}, we first need more structure to define Möbius functions, and then some finiteness conditions to take homotopy cardinality.

\begin{lem}\label{rightcomplete}
    The comodule configuration $\ds{B}_{0, \bullet} \to \ds{C}$ is complete.
\end{lem}

\begin{proof}
The right pointing is given by the extra degeneracy map $s_{-1} = \id_{\ds{I}_i} \times s_{0}$ that the décalage always have, and is thus a section to $d_0$.
It is also a monomorphism since $s_0$ is a monomorphism (the decomposition space $\ds{C}$ of
layered finite posets is complete).
\end{proof}

    The right completeness of Lemma~\ref{rightcomplete} holds for any bisimplicial groupoid of the form $X \Boxprod \Dec_\bot Y$ such that that $X$ is a Segal space and $Y$ a complete decomposition space. In contrast, the left completeness of the following lemma requires further structure, namely the extra top degeneracy map which we can define in this specific example (and which is a section to the new top face map but not to the old).
  

\begin{lem}
    The comodule configuration $\ds{B}_{\bullet,0} \to \ds{I}$ is complete.
\end{lem}

\begin{proof}
We provide a left pointing: define a new extra (vertical) degeneracy map $t_{\top+1} : \ds{B}_{i,j} \to \ds{B}_{i+1,j}$ for $i \ge -1$ in the following way:
we move the discrete part of the bottom layer of the poset into a new top layer of the set.
It is a section to $e_\top$.
It is also a monomorphism: the fibre $F_{(S,P)}$ of $t_{\top+1}$ over $(S \xra{a} \underline{i{+}1},P \xra{b} \underline{j{+}1})$ is given by the pullback
\begin{center}
    \begin{tikzcd}
        F_{(S,P)} \arrow[r, ""] \arrow[dr, phantom, "\lrcorner", very near start]
           \arrow[d, ""'] & \ds{B}_{i,j} \arrow[d, "t_{\top+1}"]\\
            1 \arrow[r, "\name{(S,P)}"'] & \ds{B}_{i+1,j}.
    \end{tikzcd}
\end{center}
If $P$ is a $(j{+}1)$-layered poset such that the bottom layer has an non empty discrete part, then the fibre is empty.
Otherwise, the fibre consists of pairs $(S', P')$ such that $S' \eq a^{-1}(\{1, \dots i\})$, the discrete part $dP'_1$ of the  bottom layer of the poset $P'$ is isomorphic to the last layer of the set $S$, that is $dP'_1 \eq a^{-1}(\{i+1\})$, and $P' - dP'_1 \eq P$. There can only be one morphism, and the fibre is then contractible.
\end{proof}

\begin{prop}
    The bicomodule configuration $\ds{B}$ is Möbius.
\end{prop}

\begin{proof}
The groupoid $\ds{B}_{0,0}$ of $1$-layered finite posets is locally finite since each finite poset has only finitely many automorphisms. The maps $s_{-1}$ and $t_{\top+1}$ are finite (since they are monomorphisms as seen above).
The maps $d_0$ and $e_\top$ are finite because each finite poset has only a finite number of $2$-layerings.

The groupoid $\ora{\ds{B}}_{0,n}$ consists of simplices which are not in the image of a degeneracy map, including the new one $s_{-1}$. In the present poset case, a simplex is an $(n{+}1)$-layered poset with no empty layers.
Similarly, a simplex in $\ora{\ds{B}}_{n,0}$ is an $n$-layered set and a poset, such that the set does not have empty layers and the poset does not have an empty discrete part.
Finally, for all $P \in \ds{B}_{0,0}$, the fibres of $d_0^{(n)}: \ora{\ds{B}}_{0,n} \to \ds{B}_{0,0}$ and $e_\top^{(n)} : \ora{\ds{B}}_{n,0} \to \ds{B}_{0,0}$
are empty for $n$ big enough since layered posets in $\ora{\ds{B}}_{0,n}$, or layered sets in $\ora{\ds{B}}_{n,0}$, do not have empty layers.
\end{proof}

The augmented bisimplicial groupoid of layered sets and posets is a Möbius bicomodule configuration.
The main contribution of the present paper is the fact that the following formula can be derived from the generalised Rota formula of Theorem~\ref{genRota}.

\begin{thm}\label{formula}
    The Möbius function of the incidence algebra of the decomposition space $\ds{C}$ of finite posets is 
    \[
    \mu(P) = 
    \begin{cases}
       (-1)^n &\text{ if $P \in \ds{C}_1$ is a discrete poset with $n$ elements}\\
         0    &\text{ else.}
    \end{cases}
    \]
\end{thm}

\begin{proof}
The left coaction $\gamma_l: \Grpd_{/{\ds{B}_{00}}} \to \Grpd_{/\ds{I}_{1}} \otimes \Grpd_{/{\ds{B}_{0,0}}}$ is given by the span
\[\ds{B}_{0,0}  \xleftarrow{e_1} \ds{B}_{1,0} 
    \xrightarrow{(u,e_0)} \ds{I}_{1} \times \ds{B}_{0,0},\]
where $u$ deletes the last layer, $e_0$ deletes the first layer, and $e_1$ joins the two layers.
The following picture represents elements in the corresponding groupoids.
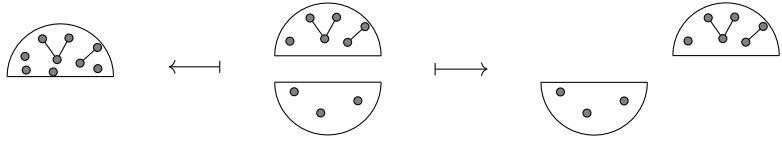
\begin{center}
\begin{tikzpicture}[x=0.7cm,y=0.7cm]
\draw [shift={(0,0)}]  plot[domain=0:3.141592653589793,variable=\t]({1*1*cos(\t r)+0*1*sin(\t r)},{0*1*cos(\t r)+1*1*sin(\t r)});
\draw [shift={(0,-0.5)}]  plot[domain=3.141592653589793:6.283185307179586,variable=\t]({1*1*cos(\t r)+0*1*sin(\t r)},{0*1*cos(\t r)+1*1*sin(\t r)});
\draw  (-1,-0.5)-- (1,-0.5);
\draw  (-1,0)-- (1,0);
\draw  (-0.33507039454439097,0.7155329229551461)-- (-0.060062480287099776,0.3216026673974046);
\draw  (-0.060062480287099776,0.3216026673974046)-- (0.16291690965124442,0.7303982156177022);
\draw  (0.37103100692703234,0.2547088504159014)-- (0.6980674455032706,0.5520147036670269);
\draw [|->] (-2.0218566133995752,-0.20885034593266605) -- (-2.9882636770331765,-0.20885034593266605);
\draw [|->] (2.0048394850737634,-0.25277793973419327) -- (3.0005316112417164,-0.25277793973419327);
\draw [shift={(5.002652816704307,-0.503313826753555)}]  plot[domain=3.141592653589793:6.283185307179586,variable=\t]({1*1*cos(\t r)+0*1*sin(\t r)},{0*1*cos(\t r)+1*1*sin(\t r)});
\draw  (4.002652816704307,-0.503313826753555)-- (6.002652816704307,-0.503313826753555);
\draw [shift={(7.480803891134185,0.0035847926872848435)}]  plot[domain=0:3.141592653589793,variable=\t]({1*1*cos(\t r)+0*1*sin(\t r)},{0*1*cos(\t r)+1*1*sin(\t r)});
\draw  (6.480803891134184,0.0035847926872848435)-- (8.480803891134187,0.0035847926872848435);
\draw  (7.145733496589793,0.7191177156424309)-- (7.420741410847084,0.32518746008468946);
\draw  (7.420741410847084,0.32518746008468946)-- (7.643720800785427,0.7339830083049872);
\draw  (7.851834898061217,0.25829364310318625)-- (8.178871336637457,0.5555994963543118);
\draw [shift={(-5.027139919287373,-0.39413899395936436)}]  plot[domain=0:3.141592653589793,variable=\t]({1*1*cos(\t r)+0*1*sin(\t r)},{0*1*cos(\t r)+1*1*sin(\t r)});
\draw  (-6.0271399192873725,-0.39413899395936436)-- (-4.027139919287373,-0.39413899395936436);
\draw  (-5.362210313831764,0.32139392899578184)-- (-5.087202399574473,-0.07253632656195969);
\draw  (-5.087202399574473,-0.07253632656195969)-- (-4.864223009636128,0.33625922165833794);
\draw  (-4.656108912360341,-0.13943014354346295)-- (-4.329072473784102,0.15787570970766257);
\draw [fill=gray] (-0.6330445319816581,-0.6814187853623229) circle (1.5pt);
\draw [fill=gray] (-0.13505722778602272,-1.0827816872513425) circle (1.5pt);
\draw [fill=gray] (0.5636115273541225,-0.8523696509817201) circle (1.5pt);
\draw [fill=gray] (-0.33507039454439097,0.7155329229551461) circle (1.5pt);
\draw [fill=gray] (-0.060062480287099776,0.3216026673974046) circle (1.5pt);
\draw [fill=gray] (0.16291690965124442,0.7303982156177022) circle (1.5pt);
\draw [fill=gray] (0.37103100692703234,0.2547088504159014) circle (1.5pt);
\draw [fill=gray] (0.6980674455032706,0.5520147036670269) circle (1.5pt);
\draw [fill=gray] (-0.714135357439576,0.27700678940973567) circle (1.5pt);
\draw [fill=gray] (4.369608284722649,-0.6847326121158779) circle (1.5pt);
\draw [fill=gray] (4.8675955889182845,-1.086095514004898) circle (1.5pt);
\draw [fill=gray] (5.566264344058429,-0.855683477735275) circle (1.5pt);
\draw [fill=gray] (7.145733496589793,0.7191177156424309) circle (1.5pt);
\draw [fill=gray] (7.420741410847084,0.32518746008468946) circle (1.5pt);
\draw [fill=gray] (7.643720800785427,0.7339830083049872) circle (1.5pt);
\draw [fill=gray] (7.851834898061217,0.25829364310318625) circle (1.5pt);
\draw [fill=gray] (8.178871336637457,0.5555994963543118) circle (1.5pt);
\draw [fill=gray] (6.766668533694608,0.2805915820970205) circle (1.5pt);
\draw [fill=gray] (-5.667527723871605,-0.27101819288541507) circle (1.5pt);
\draw [fill=gray] (-5.159715570015579,-0.30644694780560217) circle (1.5pt);
\draw [fill=gray] (-4.321813372252499,-0.2426939222024313) circle (1.5pt);
\draw [fill=gray] (-5.362210313831764,0.32139392899578184) circle (1.5pt);
\draw [fill=gray] (-5.087202399574473,-0.07253632656195969) circle (1.5pt);
\draw [fill=gray] (-4.864223009636128,0.33625922165833794) circle (1.5pt);
\draw [fill=gray] (-4.656108912360341,-0.13943014354346295) circle (1.5pt);
\draw [fill=gray] (-4.329072473784102,0.15787570970766257) circle (1.5pt);
\draw [fill=gray] (-5.691146893818398,-0.011207323470704805) circle (1.5pt);
\end{tikzpicture}
\end{center}
This left coaction $\gamma_l$ splits a $1$-layered poset into a $1$-layered set and a $1$-layered poset.

The right coaction $\gamma_r: \Grpd_{/{\ds{B}_{00}}} \to \Grpd_{/{\ds{B}_{0,0}}} \otimes \Grpd_{/\ds{C}_{1}}$ is given by the span
\[\ds{B}_{0,0}  \xleftarrow{d_0} \ds{B}_{0,1} 
    \xrightarrow{(d_1,v)} \ds{B}_{0,0} \times \ds{C}_{1},\]
where $d_1$ deletes the last layer, $v$ deletes the first layer, and $d_0$ joins the two layers.
The following picture represents elements in the corresponding groupoids.
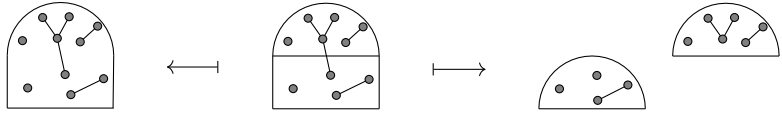
\begin{center}
\begin{tikzpicture}[x=0.7cm,y=0.7cm]
\draw [shift={(0,0)}]  plot[domain=0:3.141592653589793,variable=\t]({1*1*cos(\t r)+0*1*sin(\t r)},{0*1*cos(\t r)+1*1*sin(\t r)});
\draw  (-1,0)-- (1,0);
\draw  (-0.33507039454439097,0.7155329229551461)-- (-0.060062480287099776,0.3216026673974046);
\draw  (-0.060062480287099776,0.3216026673974046)-- (0.16291690965124442,0.7303982156177022);
\draw  (0.37103100692703234,0.2547088504159014)-- (0.6980674455032706,0.5520147036670269);
\draw [|->] (-2.0218566133995752,-0.20885034593266605) -- (-2.9882636770331765,-0.20885034593266605);
\draw [|->] (2.0048394850737634,-0.25277793973419327) -- (3.0005316112417164,-0.25277793973419327);
\draw  (-1,0)-- (-1,-1);
\draw  (1,-1)-- (-1,-1);
\draw  (1,-1)-- (1,0);
\draw  (0.08733605190851508,-0.36490821962635805)-- (-0.060062480287099776,0.3216026673974046);
\draw  (0.8102764549536944,-0.4430639388744853)-- (0.19480016587469037,-0.7459173509609783);
\draw [shift={(-4.987758028589019,0.013263548865066588)}]  plot[domain=0:3.141592653589793,variable=\t]({1*1*cos(\t r)+0*1*sin(\t r)},{0*1*cos(\t r)+1*1*sin(\t r)});
\draw  (-5.3228284231334095,0.7287964718202128)-- (-5.047820508876118,0.334866216262471);
\draw  (-5.047820508876118,0.334866216262471)-- (-4.824841118937774,0.7436617644827689);
\draw  (-4.616727021661986,0.2679723992809678)-- (-4.289690583085748,0.5652782525320934);
\draw  (-5.987758028589019,0.013263548865066711)-- (-5.987758028589018,-0.9867364511349334);
\draw  (-3.9877580285890186,-0.9867364511349334)-- (-5.987758028589018,-0.9867364511349334);
\draw  (-4.900421976680503,-0.3516446707612916)-- (-5.047820508876118,0.334866216262471);
\draw  (-4.177481573635324,-0.4298003900094186)-- (-4.792957862714328,-0.7326538020959118);
\draw [shift={(5,-1)}]  plot[domain=0:3.141592653589793,variable=\t]({1*1*cos(\t r)+0*1*sin(\t r)},{0*1*cos(\t r)+1*1*sin(\t r)});
\draw  (4,-1)-- (6,-1);
\draw  (5.672868239454502,-0.5493246707351692)-- (5.106075713171432,-0.8403802923399887);
\draw [shift={(7.513842668006751,-0.0009442800461938461)}]  plot[domain=0:3.141592653589793,variable=\t]({1*1*cos(\t r)+0*1*sin(\t r)},{0*1*cos(\t r)+1*1*sin(\t r)});
\draw  (6.513842668006751,-0.0009442800461938461)-- (8.51384266800675,-0.0009442800461938461);
\draw  (7.17877227346236,0.7145886429089523)-- (7.453780187719651,0.32065838735121077);
\draw  (7.453780187719651,0.32065838735121077)-- (7.676759577657995,0.7294539355715084);
\draw  (7.884873674933785,0.25376457036970756)-- (8.21191011351002,0.551070423620833);
\draw  (-3.9877580285890186,0.013263548865066588)-- (-4,-0.9867364511349334);
\draw [fill=gray] (-0.33507039454439097,0.7155329229551461) circle (1.5pt);
\draw [fill=gray] (-0.060062480287099776,0.3216026673974046) circle (1.5pt);
\draw [fill=gray] (0.16291690965124442,0.7303982156177022) circle (1.5pt);
\draw [fill=gray] (0.37103100692703234,0.2547088504159014) circle (1.5pt);
\draw [fill=gray] (0.6980674455032706,0.5520147036670269) circle (1.5pt);
\draw [fill=gray] (-0.714135357439576,0.27700678940973567) circle (1.5pt);
\draw [fill=gray] (-0.6160654213246324,-0.6189143071827722) circle (1.5pt);
\draw [fill=gray] (0.08733605190851508,-0.36490821962635805) circle (1.5pt);
\draw [fill=gray] (0.8102764549536944,-0.4430639388744853) circle (1.5pt);
\draw [fill=gray] (0.19480016587469037,-0.7459173509609783) circle (1.5pt);
\draw [fill=gray] (-5.3228284231334095,0.7287964718202128) circle (1.5pt);
\draw [fill=gray] (-5.047820508876118,0.334866216262471) circle (1.5pt);
\draw [fill=gray] (-4.824841118937774,0.7436617644827689) circle (1.5pt);
\draw [fill=gray] (-4.616727021661986,0.2679723992809678) circle (1.5pt);
\draw [fill=gray] (-4.289690583085748,0.5652782525320934) circle (1.5pt);
\draw [fill=gray] (-5.701893386028595,0.2902703382748025) circle (1.5pt);
\draw [fill=gray] (-5.603823449913651,-0.6056507583177057) circle (1.5pt);
\draw [fill=gray] (-4.900421976680503,-0.3516446707612916) circle (1.5pt);
\draw [fill=gray] (-4.177481573635324,-0.4298003900094186) circle (1.5pt);
\draw [fill=gray] (-4.792957862714328,-0.7326538020959118) circle (1.5pt);
\draw [fill=gray] (4.3867728624493525,-0.6224135038082472) circle (1.5pt);
\draw [fill=gray] (5.090174335682499,-0.36840741625183293) circle (1.5pt);
\draw [fill=gray] (5.672868239454502,-0.5493246707351692) circle (1.5pt);
\draw [fill=gray] (5.106075713171432,-0.8403802923399887) circle (1.5pt);
\draw [fill=gray] (7.17877227346236,0.7145886429089523) circle (1.5pt);
\draw [fill=gray] (7.453780187719651,0.32065838735121077) circle (1.5pt);
\draw [fill=gray] (7.676759577657995,0.7294539355715084) circle (1.5pt);
\draw [fill=gray] (7.884873674933785,0.25376457036970756) circle (1.5pt);
\draw [fill=gray] (8.21191011351002,0.551070423620833) circle (1.5pt);
\draw [fill=gray] (6.799707310567174,0.2760625093635418) circle (1.5pt);
\end{tikzpicture}
\end{center}
This right coaction $\gamma_r$ splits a $1$-layered poset into two $1$-layered posets.

    Computing the right-hand side of the formula of Theorem~\ref{genRota}, we obtain
\[(|\delta^{L}| \star_r |\mu^Y| )(P)= |\mu^Y|(P), \text{ for all poset } P \in \ds{B}_{0,0}.\]
Indeed, $|\delta^{L}|$ is different from $0$ only if evaluated on the empty poset, and we are in this situation only if the $2$-layered poset in $\ds{B}_{0,1}$ consists of an empty first layer, and a second layer with the whole original poset.
The left-hand side gives
\[(|\mu^{X}| \star_l |\delta^R|) (P)= 
\begin{cases}
    |\mu^X|(P) & \text{ if the poset $P$ is discrete}\\
       0  & \text{ otherwise.}
\end{cases}
\]
Indeed, $|\delta^{R}|$ is different from $0$ only for the empty set, and we are in this situation only if the $2$-layered poset in $\ds{B}_{1,0}$ consists of an empty second layer, and a first discrete layer (that is a set) in $\ds{I}_1$.
We conclude by recalling that the Möbius function $\mu$ of a set $S$ with $n$ elements is given by $\mu(S) = (-1)^n$.
\end{proof}

\subsection*{Möbius function of any directed restriction species}\label{sec:generalised}

We have treated the case of the decomposition space of finite posets, corresponding to the terminal directed restriction species.
A \emph{directed restriction species} is a groupoid-valued presheaf on the category $\mathbb{C}$ of finite posets and convex maps.
Every directed restriction species $R: \mathbb{C}\op \to \Grpd$ defines a decomposition space $\ds{R}$ (and hence a coalgebra); we refer to \cite{GKT:restrict} for all details.
This decomposition space comes equipped with a culf functor $\ds{R} \to \ds{C}$ \cite[Lemma 7.6]{GKT:restrict}. It follows that $\ds{R}$ is complete, locally finite, locally discrete, and of locally finite length, and is in particular a Möbius decomposition space.
Examples of directed restrictions species include rooted forests and directed graphs of various kinds \cite{GKT:restrict}.
Having computed the Möbius function for the decomposition space $\ds C$ in the previous section, we can now obtain it for all directed restriction species.

\begin{cor}\label{formulaDRS}
    The Möbius function of the incidence algebra of the decomposition space $\ds R$ associated to a directed restriction species $R : \mathbb C\op \to \Grpd$ is 
    \[
        \mu(Q) = 
    \begin{cases}
       (-1)^n &\text{ if the underlying poset of $Q \in \ds R_1$ is discrete with $n$ elements}\\
         0    &\text{ else.}
    \end{cases}
    \]
\end{cor}

\begin{proof}
Once an expression has been found for the Möbius function of the decomposition space $\ds{C}$, it can be pulled back to $\ds{R}$ along the culf functor $\ds{R} \to \ds{C}$ to obtain the corresponding expression for the Möbius function of $\ds{R}$, as done in \cite{GKT:combinatorics}.
\end{proof}

Note that an ordinary restriction species in the sense of Schmitt \cite{Schmitt} is a special case of a directed restriction species, namely one supported on discrete posets \cite{GKT:restrict}. The Möbius function then
reduces to the well-known formula $(-1)^n$ for an underlying set with $n$ elements, see \cite[\S 3.3.10]{GKT:combinatorics}.

\begin{cor}\label{trees}
The Möbius function of the Butcher--Connes--Kreimer Hopf algebra of rooted forests
is
    \[
        \mu(F) = 
    \begin{cases}
       (-1)^n &\text{ if $F$ consists of $n$ isolated root nodes}\\
         0    &\text{ else.}
    \end{cases}
    \]
\end{cor}

\begin{proof}
    Rooted forests form an example of directed restriction species \cite[\S 7.12]{GKT:restrict}: a rooted forest has an underlying poset, with induced rooted-forest structure on convex subposets. The resulting bialgebra is the Butcher--Connes--Kreimer Hopf algebra \cite[\S 2.2]{GKT:restrict}.
\end{proof}

Finally, we obtain the Möbius function of the incidence bialgebra of $P$-trees, for any finitary polynomial endo\-functor $P: \Grpd_{/I} \to \Grpd_{/I}$, that is, given by a diagram of groupoids $I \to E \xra{p} B \to I$ such that the fibres of $E \xra{p} B$ are finite.
A $P$-tree is a tree with edges decorated in $I$, and nodes decorated in $B$; we refer to \cite{GKT:faa} for a precise definition and examples.
Note that allowing the nodeless tree, the notion of forests do not form a directed restriction species \cite[\S 7.12]{GKT:restrict}.

\begin{cor}\label{mobiusPtrees}
    The Möbius function of the incidence bialgebra of $P$-trees (for any finitary polynomial endo\-functor $P$) is
    \[
        \mu(T) = 
    \begin{cases}
       (-1)^n &\text{ if $T$ consists of $n$ $P$-corollas and possibly isolated edges}\\
         0    &\text{ else.}
    \end{cases}
    \]
\end{cor}

\begin{proof}
The free monad $F$ on a finitary polynomial endofunctor $P$ is the polynomial monad represented by 
\[
    A \xla{} T'_P \xra{q} T_P \xra{} A,
\]
where $T_P$ is the groupoid of $P$-trees,  $T'_P$ is the set of isomorphism classes of $P$-trees with a marked leaf, the left map returns the decoration of the marked leaf, the right map returns the decoration of the root, and the middle map forgets the mark.
Operads can be seen as certain polynomial monads \cite[\S 2.6]{KW}.
The operations are the $P$-corollas.
The factorisations of operations correspond to cuts in trees.
The bialgebra of $P$-trees is the incidence bialgebra of the free monad on $P$, meaning the incidence bialgebra of the two-sided bar construction on the free monad on $P$, see \cite{KW}. 

The \emph{core} of a $P$-tree is the combinatorial tree obtained by forgetting the $P$-decoration, the leaves, and the root edge \cite{Kock:rootedtrees}.
It defines a culf functor from the bar-construction of a free monad to the decomposition space of trees \cite{KW}. In the same way as the proof of Corollary~\ref{formulaDRS}, we  pull back the expression of Corollary~\ref{trees} along the culf functor to conclude.
\end{proof}


\address
\end{document}